\documentclass[a4paper]{amsart}

\usepackage[latin1]{inputenc}
\usepackage{amsmath, amssymb, amsthm, graphicx}
\newtheorem{theorem}{Theorem}[section]
\newtheorem{lemma}[theorem]{Lemma}
\newtheorem{proposition}[theorem]{Proposition}
\newtheorem{corollary}[theorem]{Corollary}

\newtheorem{remark}[theorem]{Remark}

\def\a{\alpha}
\def\d{\delta}
\def\r{\rho}
\newcommand{\nn}{n}

\newcommand{\RR}{\mathbb{R}}
\newcommand{\e}{\varepsilon}
\newcommand{\su}{\subseteq}
\newcommand{\sm}{\setminus}
\newcommand{\D}{\Delta}

\newcommand{\C}{\mathbb{C}}
\newcommand{\R}{\mathbb{R}}

\newcommand{\Z}{\mathbb{Z}}

\newcommand{\Acal}{\mathcal{A}}
\newcommand{\Ccal}{\mathcal{C}}

\newcommand{\Scal}{\mathcal{S}}

\newcommand{\chim}{\chi_{\rm m}}
\newcommand{\lm}[1]{\lambda(#1)}

\newcommand{\olm}[1]{\overline{\lambda}(#1)}

\newcommand{\ud}{\overline{\delta}}
\newcommand{\one}{{\bf 1}}

\DeclareMathOperator{\ort}{{\rm O}}

\DeclareMathOperator{\diam}{{\rm diam}}

\title{Better bounds for planar sets avoiding unit distances}

\author{Tamás Keleti}
\address{Tamás Keleti, Institute of Mathematics,
Eötvös Loránd University,
Pázmány P. sétány 1/C,
1117 Budapest, Hungary}
\email{tamas.keleti@gmail.com}
\urladdr{http://www.cs.elte.hu/analysis/keleti}

\author{Máté Matolcsi}
\address{Máté Matolcsi, Alfréd Rényi Institute of Mathematics,
Hungarian Academy of Sciences POB 127 H-1364 Budapest, Hungary,
Tel: (+361) 483-8307, Fax: (+361) 483-8333}
\email{matomate@renyi.hu}

\author{Fernando Mário de Oliveira Filho}
\address{Fernando M. de Oliveira Filho, Instituto de Matemática e
  Estatística, Universidade de São Paulo, Rua do Matão 1010, 05508-090
  São Paulo/SP, Brazil}
\email{fmario@gmail.com}

\author{Imre Z. Ruzsa}
\address{Imre Z. Ruzsa, Alfréd Rényi Institute of Mathematics,
Hungarian Academy of Sciences POB 127 H-1364 Budapest, Hungary,
Tel: (+361) 483-8307, Fax: (+361) 483-8333}
\email{ruzsa@renyi.hu}

\date{October 26, 2015}

\thanks{Part of this research was done when T. Keleti was a visitor at
  the Alfréd Rényi Institute of Mathematics; he was also supported
  by OTKA grant no.~104178. M. Matolcsi and I.Z. Ruzsa were supported
  by OTKA no.~109789 and ERC-AdG 321104. F.M. de Oliveira Filho was
  partially supported by FAPESP project~\hbox{13/03447-6}.}



\begin{document}

\begin{abstract}
  A $1$-avoiding set is a subset of~$\R^n$ that does not contain
  pairs of points at distance~$1$.  Let~$m_1(\R^n)$ denote the maximum
  fraction of~$\R^n$ that can be covered by a measurable $1$-avoiding set. We
  prove two results. First, we show that any $1$-avoiding set
  in~$\R^n$ ($n\ge 2$) that displays block structure (i.e., is made up of blocks
such that the distance between any two points from the same
block is less than $1$ and points from distinct blocks lie farther
than~$1$ unit of distance apart from each other)
has density strictly less than~$1/2^n$. For the
  special case of sets with block structure this proves a
  conjecture of Erd\H{o}s asserting that~$m_1(\R^2) < 1/4$. Second, we use linear
  programming and harmonic analysis to show that~$m_1(\R^2) \leq
  0.258795$.
\end{abstract}

\maketitle

\markboth{T. Keleti, M. Matolcsi, F.M. de Oliveira Filho, and
  I.Z. Ruzsa}{Better bounds for planar sets avoiding unit distances}


\section{Introduction}

The \textit{unit-distance graph} of~$\R^n$ is the graph whose vertex
set is~$\R^n$ and in which~$x$ and~$y$ are adjacent if~$\|x-y\| =
1$. A well-known problem in geometry, going back to Nelson and
Hadwiger (see Soifer~\cite{Soifer2009} for a historical survey), asks for
the chromatic number~$\chi(\R^n)$ of the unit-distance graph
of~$\R^n$.

A related problem considers independent sets of the
unit-distance graph. Let~$G = (V, E)$ be a graph. A set~$I \subseteq V$
is \textit{independent} if it does not contain a pair of adjacent
vertices. The \textit{independence number} of~$G$, denoted
by~$\alpha(G)$, is the maximal cardinality of an independent set.

A set~$A \subseteq \R^n$ is independent in the unit-distance graph if it
does not contain pairs of points at distance~$1$, that is,
$\|x - y\| \neq 1$ for all~$x$, $y \in A$. We also say that~$A$
\textit{avoids distance~$1$} or that it is a \textit{$1$-avoiding
  set}.  A possible measure
for the size of an independent set in this case is its density, that
is, the fraction of space that it covers (see~§\ref{sec:prelim} for a
rigorous definition). Our aim is to estimate the
maximal
(or more precisely the least upper bound of the)
fraction of~$\R^n$ that can be covered by a measurable set that avoids
distance~$1$.

We denote this maximal fraction by~$m_1(\R^n)$. In §\ref{sec:planar},
we show that~$m_1(\R^2) \leq 0.258795$. This result is related to a
conjecture of Erd\H{o}s (cf.~Székely~\cite{Szekely2002}), stating
that~$m_1(\R^2) < 1/4$.

As for the chromatic number of the
Euclidean plane, all that is known is that~$4 \leq \chi(\R^2) \leq 7$.
The upper bound comes from a simple periodic coloring of~$\R^2$,
whereas the lower bound comes from a finite subgraph of the
unit-distance graph, the Moser spindle (cf.~Moser and
Moser~\cite{MoserM1961}), whose chromatic number is~$4$ (see
Figure~\ref{fig:moser}).

 \begin{figure}[htb]
 \begin{center}
\includegraphics{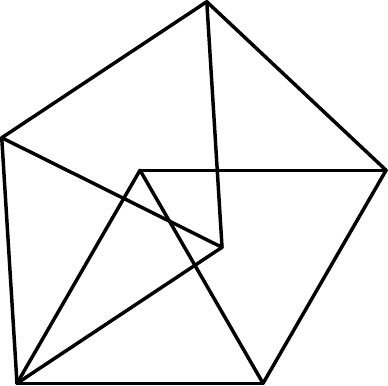}
\hskip2cm
\includegraphics{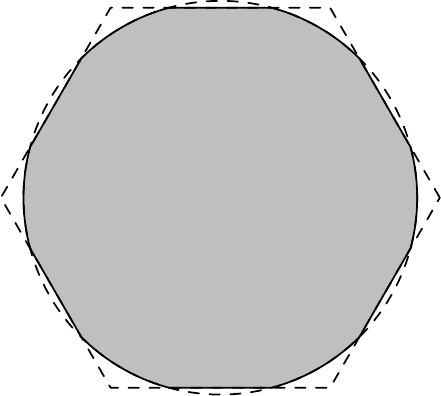}
\end{center}



\caption{On the left, the Moser spindle~\cite{MoserM1961}. Each
  segment has length exactly~$1$. On the right, the optimal tortoise
  in Croft's construction~\cite{Croft1967}.}
\label{fig:moser}
\end{figure}

In view of the difficulty of computing~$\chi(\R^n)$,
Falconer~\cite{Falconer1986} introduced the \textit{measurable chromatic
  number}, denoted by~$\chim(\R^n)$, in which the restriction is added
that the color classes must be Lebesgue-measurable sets. In other
words, one wishes to partition~$\R^n$ into the minimum possible number
of Lebesgue-measurable $1$-avoiding sets. Obviously, $\chim(\R^n) \geq
\chi(\R^n)$. Falconer proved that~$\chim(\R^2) \geq 5$. Since
\[
m_1(\R^n) \chim(\R^n) \geq 1,
\]
showing Erd\H{o}s' conjecture would give another proof
of Falconer's result.

A simple lower bound for~$m_1(\R^2)$ comes from the following
construction. Consider the hexagonal lattice with minimal vectors of
length~$2$ and place at each point of the lattice an open disk of
radius~$1/2$. This is a $1$-avoiding set of
density~$\pi / (8\sqrt{3}) = 0.2267\ldots$. A slight improvement was
given by Croft~\cite{Croft1967}. His construction is as follows.
Shrink the hexagonal lattice slightly so as to have minimal vectors of
length~$1+x$, where $x<1$, and place at each lattice point the
intersection of an open disk of radius~$1/2$ and an open regular
hexagon of height~$x$. The disks guarantee that inside each block
every distance is less than $1$, while the hexagons guarantee that
points from different blocks have distance greater than $1$.
Taking~$x=0.96553\ldots$ maximizes the density of the union of all the
blocks, showing that~$m_1(\R^2)\geq 0.22936$. This intersection of a
disk with a hexagon has often been called a \textit{tortoise};
Figure~\ref{fig:moser} shows an optimal tortoise.

Any finite subgraph~$G = (V, E)$ of the unit-distance graph of~$\R^n$
provides an upper bound for~$m_1(\R^n)$, namely~$\alpha(G) / |V|$;
this can be seen via a simple averaging argument.\footnote{This is
  related to the following observation: Let~$G$ be a subgraph of a
  finite vertex-transitive graph~$H$.  Then
  $\alpha(H) / |V(H)| \leq \alpha(G) / |V(G)|$.}  For the plane, we
then immediately have that~$m_1(\R^2) \leq 1/3$, as can be seen from
the equilateral triangle. A better bound of~$2/7$ is provided by the
Moser spindle; this is the best upper bound that has been obtained
from a finite subgraph.

Using further ideas, Székely~\cite{Szekely1984} proved a bound of
$\approx 0.279 < 2/7 \approx 0.285$. Oliveira and
Vallentin~\cite{OliveiraV2010} gave the currently best known upper
bound of $\approx 0.268$. Their method is based on a mix of linear
programming and harmonic analysis; it is a strengthening of it that
will be used in §\ref{sec:planar} to prove
that~$m_1(\R^2) \leq 0.258795$.

 Frankl and Wilson~\cite{FranklW1981} construct finite
subgraphs of the unit-distance graph of~$\R^n$ whose chromatic numbers
grow exponentially fast in~$n$. These same subgraphs can be used to
provide the asymptotic upper bound
\[
m_1(\R^n) \leq (1 + o(1)) 1.207^{-n}.
\]
via the averaging argument. The method of Oliveira and
Vallentin~\cite{OliveiraV2010} also provides an exponential upper
bound, $m_1(\R^n) \leq (1 + o(1))1.1654^{-n}$, which is somewhat
weaker than the bound of Frankl and Wilson.  A combination of both
arguments, requiring detailed analysis, was used by Bachoc, Passuello,
and Thiery~\cite{BachocPT2014} to obtain the best known asymptotic
upper bound
\[m_1(\R^n)\leq (1 + o(1)) 1.268^{-n}.
\] We remark here that the results of \cite{BachocPT2014} and the present paper do not overlap: the authors of \cite{BachocPT2014} obtain asymptotic upper bounds as $n\to \infty$, and some numerical improvements for dimensions $4\le n \le 24$.

The behavior of~$\chi(\R^n)$ under restrictions placed on the color
classes has also been studied. We have already mentioned Falconer's
measurable chromatic number. Other restrictions include requiring all
classes to be either open or closed sets --- in both cases it can be
shown that the chromatic number of~$\R^2$ is either~$6$ or~$7$
(cf.~Soifer~\cite{Soifer2009}).

Similarly, we may place some natural restrictions on $1$-avoiding sets and study their
maximal possible upper densities. In particular, we will consider sets with
block structure. We say that a set~$A$ has \textit{block structure} if
it is a union
\[
A = \bigcup_{i=0}^\infty A_i
\]
of \textit{blocks}~$A_i$, where~$\|x-y\| < 1$ if~$x$ and~$y$ belong to
the same block, and~$\|x-y\| > 1$ if~$x$ and~$y$ belong to different
blocks. Notice that we do not require $A$ to be measurable (but this will not make an essential difference;
cf. the proof of Theorem \ref{bm} below).

A set with block structure is clearly a $1$-avoiding set. All known
constructions of $1$-avoiding sets of ``high density'' are actually
constructions of sets with block structure, like  the
hexagonal lattice construction of disks, or Croft's construction. Recall that
Erd\H{o}s conjectured that~$m_1(\R^2) < 1/4$. Larman and Rogers, and
before them Moser (cf.~Larman and Rogers~\cite{LarmanR1972}), made the
following conjecture: the volume of a closed $1$-avoiding set inside a
ball of radius~$1$ in~$\R^n$ is less than~$1/2^n$ of the volume of the
ball. A simple argument shows that this conjecture
implies~$m_1(\R^n) < 1/2^n$, and it is therefore a generalization of
Erd\H{o}s' conjecture. In §\ref{sec:blocks} we will show that any
subset of~$\R^n$ ($n\ge 2$) with block structure has upper density less
than~$1/2^n$.

\subsection{Preliminaries and notation}
\label{sec:prelim}

Throughout the paper,~$\lm{A}$ and $\olm{A}$ will denote the Lebesgue measure of a
(measurable) set~$A$, and the outer measure of any set~$A$, respectively.

 Let~$A \subseteq \R^n$ be a measurable
set. We say that its \textit{density} is~$\delta(A)$ if for
all~$p \in \R^n$ we have
\[
\delta(A) = \lim_{r \to \infty} \frac{\lm{A \cap S(p, r)}}{\lm{S(p,
    r)}},
\]
where~$S(p, r)$ is the $n$-dimensional cube of side~$2r$ centered
at~$p$. For a set that has a density, the cube can be substituted by
any reasonable body, like the ball, say, without changing the
resulting density.

A measurable set may not have a density, but every set has an \textit{upper density}
\[
\ud(A) =  \limsup_{r \to \infty} \frac{\olm{A \cap
    S(p, r)}}{\lm{S(p, r)}}.
\]
We define
\[
m_1(\R^n) = \sup\{\, \ud(A) : \text{$A \subseteq \R^n$ is $1$-avoiding
  and measurable}\,\}.
\]

We say that a set~$A$ is \textit{periodic} if there is a
lattice~$L \subseteq \R^n$ that leaves~$A$ invariant, that
is,~$x + A = A$ for all~$x \in L$. Then~$L$ is the \textit{periodicity
  lattice} of~$A$.

Measurable periodic sets have densities. Moreover, a simple
argument shows that the densities of periodic $1$-avoiding sets can
come as close as desired to~$m_1(\R^n)$ (cf.~Oliveira and
Vallentin~\cite{OliveiraV2010}). So when computing upper bounds
for~$m_1(\R^n)$ we may restrict ourselves to periodic sets.


\section{Sets with block structure}
\label{sec:blocks}

In any dimension, all known examples of 1-avoiding sets of ``high
density'' are made up of disjoint blocks, i.e., they are sets of block
structure as defined in the introduction (see the end of this
section for an example of a 1-avoiding set of small positive density
which does not have block structure).

On the one hand it is natural to consider this class of sets because
human imagination of 1-avoiding sets seems to be more or less
restricted to this class. On the other hand, it seems very elusive to
prove rigorously that a 1-avoiding set of maximum or close to
maximum density must be block-structured.

The following theorem, for~$n = 2$, is a special case of the
conjecture of Erd\H{o}s (cf.~Székely~\cite{Szekely2002}) given in the
introduction; for~$n \geq 3$, it is a special case of a conjecture of
Larman, Rogers, and Moser (cf.~Larman and Rogers~\cite{LarmanR1972}),
also discussed in the introduction.

\begin{theorem}
\label{bm}
  Let $\nn\geq 2$ and let $A\subseteq\R^\nn$ be a 1-avoiding set having block structure. Then
  $\overline{\delta}(A)\leq 1/2^\nn-\e_\nn$ for some $\e_\nn>0$.
\end{theorem}

We will prove the following slightly stronger result.

\begin{theorem}
\label{bm2}
Let $n\ge 2$ and let
$A_1$, $A_2$, \dots\ $\subseteq \R^\nn$ be sets of diameter at most $1$
such that the distance of any two of them is at least $1$.
Then the upper density of $A=\bigcup_{i=1}^\infty A_i$ is
at most $1/2^\nn-\e_\nn$ for some $\e_\nn>0$.
\end{theorem}

Observe that both properties (diameter at most $1$,  distance at least $1$)
are preserved if we replace the sets by their closure, so we may assume that
the sets $A_i$ are all closed. This also makes them measurable.
(This reduction works only under the assumption of block structure.
In the general case, it is fairly easy to show the existence of a
(non-measurable) 1-avoiding set of full outer measure, that is, such that the inner measure of
its complement is 0.)

First we present a simple proof for the weaker statement
when $\e_\nn$ is discarded, that is,
$\overline{\delta}(A)\leq 1/2^\nn$.

Let $B_r$ denote the open ball of radius $r$ around
the origin and let
\[
C_i=A_i+B_{1/2}=\{\,a+b : a\in A_i,\ b\in B_{1/2}\,\}.
\]
It is clear from the assumptions that $C_i\cap C_j =\emptyset$, for
all $i\neq j$.



Applying the Brunn-Minkowski inequality (see e.g. \cite[Theorem 4.1]{Gardner2002}) to the  sets
$ {A}_i$, $ {B}_{1/2}$ we obtain
$\lm{C_i}^{1/\nn}\geq \lm{ {A}_i}^{1/\nn}+\lm{B_{1/2}}^{1/\nn}$. Furthermore
 the
isodiametric inequality gives $\lm{ {A}_i}\leq \lm{B_{1/2}}$. By
combining these inequalities we get
\begin{equation}\label{eq:bm}
\frac{\lm{A_i}^{1/\nn}}{\lm{C_i}^{1/\nn}}\leq
\frac{\lm{ {A}_i}^{1/\nn}}{\lm{ {A}_i}^{1/\nn}+\lm{B_{1/2}}^{1/\nn}} =
\frac1{1+\left(\frac{\lm{B_{1/2}}}{\lm{ {A}_i}}\right)^{1/\nn}} \leq
\frac1{1+1}=
\frac12.
\end{equation}
Since the sets $C_i$ are pairwise disjoint, $A_i\su C_i$, and
$A=\bigcup_i A_i$, this shows that
$\overline{\delta}(A)\leq 1/2^\nn$.

The plan to show $\overline{\delta}(A)\leq 1/2^\nn-\e_\nn$ is the
following: If $\overline{\delta}(A)$ is close to $1/2^\nn$ then in the
above argument we must have that for most~$i$ the isodiametric
inequality is almost an equality and~$\diam A_i$ is close to~$1$. By a
stability theorem this implies that each such ${A_i}$ is very
close to a ball of radius~$1/2$ and then most of the sets
$C_i=A_i+B_{1/2}$ are very close to unit balls.  But the density of
any unit ball packing is well-separated from~$1$, and then so is the
density of~$\bigcup_i C_i$.  Since~$A$ has density at most~$1/2^\nn$
in~$\bigcup_i C_i$, this implies that the density of~$A$ is well
separated from $1/2^\nn$.

The stability result we use is the following theorem of Maggi,
Ponsiglione, and Pratelli~\cite{MaggiPP2014}.

\begin{lemma}\label{t:stability}
  Let $E\su \R^\nn$ be a measurable set with $\lm{E}>0$ and $\diam E=2$.
  Then there exist $x$, $y\in\R^\nn$ such that
$$
E \su B(x, 1+r) \qquad \textrm{and} \qquad B(y,1) \su E + B_r,
$$
where $B(z,R)$ denotes the ball centered at $z$ with radius $R$ and
$$
r=K_\nn \left(\frac{\lm{B_1}}{\lm{E}}-1\right)^{1/\nn}
$$
for some constant~$K_\nn$ that depends only on~$\nn$.
\end{lemma}

Note that for sets of diameter~$2$ the expression $\lm{B_1}/\lm{E}-1$, which
is called \textit{isodiametric deficit} by Maggi, Ponsiglione, and
Pratelli, is nonnegative by the isodiametric inequality and expresses
the error in that inequality.

We will need the following simple corollary of the above result.

\begin{corollary}\label{c:stability}
  For any $\nn\geq 2$ there exists an increasing function
  $\beta = \beta_\nn\colon (0,\infty)\to(0,\infty)$ with
  $\lim_{\r\to 0}\beta_\nn(\r)=0$ with the following property.  For
  every measurable $E\su\R^\nn$ with $\lm{E}>0$ and $\diam E\leq 1$ there
  exist $x$, $y\in\R^\nn$ such that
$$
E \su B(x, 1/2 + \beta(\r(E))) \qquad \textrm{and} \qquad
B(y,1/2) \su E + B_{\beta(\r(E))},
$$
where
$$
\r(E)=\frac{\lm{B_{1/2}}}{\lm{E}}-1.
$$
\end{corollary}

\begin{proof}
  By rescaling Lemma~\ref{t:stability} we get that for any
  measurable $E\su\R^\nn$ with $\lm{E}>0$ and $\diam E < \infty$ there
  exist $x,y\in\R^\nn$ such that
\begin{equation}\label{eq:rescaled}
E \su B\left(x, \frac{\diam E}2+r\right)
\qquad \textrm{and} \qquad
B\left(y,\frac{\diam E}2\right) \su E + B_r,
\end{equation}
where
$$
r=K_\nn \frac{\diam E}2\left(\frac{\lm{B_{(\diam E)/2}}}{\lm{E}}-1\right)^{1/\nn}.
$$

If $\diam E\leq 1$ then
\begin{equation}\label{eq:r}
r=K_\nn \frac{\diam E}2\left(\frac{\lm{B_{(\diam E)/2}}}{\lm{E}}-1\right)^{1/\nn} \leq
\frac{K_\nn}2 \r(E)^{1/\nn},
\end{equation}
so
\begin{equation}\label{eq:contained}
E \su B\left(x, \frac{\diam E}2+r\right)
\su B\left(x, \frac12+\frac{K_\nn}2 \r(E)^{1/\nn}\right).
\end{equation}

By the isodiametric inequality we have $\lm{E}\leq \lm{B_{(\diam E)/2}}$, so
$$
\r(E)=\frac{\lm{B_{1/2}}}{\lm{E}}-1\geq
\frac{\lm{B_{1/2}}}{ \lm{B_{(\diam E)/2}}}-1
=\frac1{(\diam E)^\nn}-1,
$$
hence
\begin{equation}\label{eq:diam}
\diam E \geq \left(\frac1{1+\rho(E)}\right)^{1/\nn}.
\end{equation}

From the second part of (\ref{eq:rescaled}), using $B(z,a)+B_b=B(z,a+b)$
we get that
\begin{equation}\label{eq:contains}
B(y,1/2) \su E + B_{r+(1-\diam E)/2}.
\end{equation}

Then (\ref{eq:r}), (\ref{eq:contained}), (\ref{eq:diam}) and
(\ref{eq:contains}) show that the function
$$
\beta_\nn(\r)=\frac{K_\nn}2 \r^{1/\nn} +
\frac{1}{2}\left(1-\left(\frac1{1+\r}\right)^{1/\nn}\right)
$$
has all the required properties.
\end{proof}

Now we  prove Theorem~\ref{bm2}.

\begin{proof}[Proof of Theorem~\ref{bm2}]
We fix the dimension~$\nn$; all constants may depend on it. We will
show that if~$N$ is large enough then the density of $A$ in
$[-N,N]^\nn$ is at most $1/2^\nn-\e_\nn$ for some positive $\e_\nn$.

Suppose without loss of generality that the blocks $A_i$ are
enumerated so that $A_1$,~\dots,~$A_m$ are the ones that have nonempty
intersection with $[-N,N]^\nn$ and so
\begin{equation}\label{eq:n}
A\cap [-N,N]^\nn = \bigcup_{i=1}^m (A_i \cap [-N,N]^\nn).
\end{equation}
 Let
$$
\r_i=\r({A}_i)=\frac{\lm{B_{1/2}}}{\lm{{A}_i}}-1.
$$
Then using (\ref{eq:bm}) we get
\begin{equation}\label{eq:rho}
\frac{\lm{A_i}^{1/\nn}}{\lm{C_i}^{1/\nn}}\leq
\frac1{1+\left(\frac{\lm{B_{1/2}}}{\lm{{A}_i}}\right)^{1/\nn}}=
\frac1{1+(\r_i+1)^{1/\nn}}.
\end{equation}

Let $\a$ and $\r$ be two small positive constants that will be
specified later.  Let
$$
I=\{\, i\in\{1,\ldots,n\} : \r_i \geq \r\, \} \qquad \textrm{and} \qquad
J = \{1,\ldots,n\} \sm I.
$$

By~\eqref{eq:rho}, for every $i\in I$ we have
$$
\lm{A_i} \leq \lm{C_i} \left(\frac1{1+(\r+1)^{1/\nn}}\right)^\nn.
$$
By (\ref{eq:bm}), for every $i \in J$ we have
$\lm{A_i} \leq \lm{C_i} / 2^\nn$.  Therefore we have
\begin{align*}
\lm{A\cap [-N,N]^\nn} &\leq
\sum_{i\in I} \lm{C_i} \left(\frac1{1+(\r+1)^{1/\nn}}\right)^\nn
+ \sum_{i\in J} \frac{\lm{C_i}}{2^\nn}\\
&\leq \sum_{i=1}^m \frac{\lm{C_i}}{2^\nn} - \sum_{i\in I} \lm{C_i}
\left(\frac1{2^\nn}-\left(\frac1{1+(\r+1)^{1/\nn}}\right)^\nn\right) .
\end{align*}

First we consider the case when
$$
\sum_{i\in I} \lm{C_i} \geq \a (2N)^\nn.
$$
Using that the sets $C_i$ are pairwise disjoint and
$\bigcup_{i=1}^m C_i \su [-N-2,N+2]^\nn$, we get
\begin{equation}\label{eq:case1}
\frac{\lm{A\cap[-N,N]^\nn}}{\lm{[-N,N]^\nn}} \leq
\frac{(2(N+2))^\nn}{(2N)^\nn} \frac1{2^\nn} -
\a\left(\frac1{2^\nn}-\left(\frac1{1+(\r+1)^{1/\nn}}\right)^\nn\right).
\end{equation}

Next consider the case when
\begin{equation}\label{eq:case2}
\sum_{i\in I} \lm{C_i} < \a (2N)^\nn.
\end{equation}
By definition for any $i\in J$ we have $\r({A}_i)<\r$.  By
Corollary~\ref{c:stability} this implies that for any $i\in J$ there
exist $x_i$ and $y_i$ such that
\begin{equation}\label{eq:inout}
{A}_i \su B(x_i, 1/2+\beta(\r)) \qquad \textrm{and} \qquad
B(y_i,1/2) \su {A}_i + B_{\beta(\r)}.
\end{equation}
Since $\lim_{\r\to 0}\beta(\r)=0$ we
can guarantee $\beta(\r)<1/2$ by taking $\r$ small enough.

Note that the first part of (\ref{eq:inout}) implies
\begin{equation}\label{eq:1plusbeta}
C_i = A_i + B_{1/2} \su B(x_i,1+\beta(\r))
\end{equation}
and the second part of (\ref{eq:inout}) implies
\begin{equation}\label{eq:1minusbeta}
B(y_i,1-\beta(\r)) \su A_i + B_{1/2} = C_i.
\end{equation}
Since the sets $C_i$ are pairwise disjoint, this implies that the balls
$B(y_i,1-\beta(\r))$, for~$i\in J$, are pairwise disjoint.

Let $\D_\nn$ be an upper bound on the density of the union of disjoint
balls of the same size in $\R^\nn$.  Although the best $\D_\nn$ is not
known for general~$\nn$, it is known
that $\D_\nn<1$ for $\nn\geq 2$.


Thus the density of $\bigcup_{i\in J} B(y_i,1-\beta(\r))$ in
$[-N-2,N+2]^\nn$ is at most $\D_\nn$, that is,
\begin{equation}
\label{eq:delta-bound}
\frac{1}{(2(N+2))^\nn}\sum_{i \in J} \lm{B(y_i, 1 - \beta(\r))} \leq \Delta_\nn.
\end{equation}

Notice
that~$\lm{B(x_i, 1 + \beta(\rho))} / \lm{B(y_i, 1 - \beta(\rho))} = ((1 +
\beta(\rho)) / (1 - \beta(\rho)))^\nn$.
Let~$\Acal_J = \bigcup_{i\in J} A_i$.
Then, using first \eqref{eq:bm} then \eqref{eq:1plusbeta}
and finally \eqref{eq:delta-bound}
we get that
\begin{equation}\label{eq:Jterms}
\begin{split}
\frac{\lm{\Acal_J}}{\lm{[-N,N]^\nn}} \leq \frac{1}{(2N)^\nn}\sum_{i \in J} \frac{\lm{C_i}}{2^\nn}
&\leq \frac{1}{(2N)^\nn}\sum_{i \in J} \frac{\lm{B(x_i, 1 + \beta(\rho))}}{2^\nn}\\
&\leq \frac{(2(N+2))^\nn}{(2N)^\nn} \cdot \Delta_\nn \cdot
\left(\frac{1+\beta(\rho)}{1-\beta(\rho)}\right)^\nn \cdot
\frac{1}{2^\nn}.
\end{split}
\end{equation}

Let~$\Acal_I = \bigcup_{i \in I} A_i$. By (\ref{eq:bm}) and
(\ref{eq:case2}) we have
\begin{equation}\label{eq:Iterms}
\frac{\lm{\Acal_I}}{\lm{[-N,N]^\nn}} \leq \frac{\a}{2^\nn}.
\end{equation}
Combining (\ref{eq:Jterms}) and (\ref{eq:Iterms}) we get in
this case
\begin{equation}\label{eq:case2estimate}
\frac{\lm{A\cap[-N,N]^\nn}}{\lm{[-N,N]^\nn}} \leq
\frac{(2(N+2))^\nn}{(2N)^\nn}\cdot \D_\nn
\cdot \left(\frac{1+\beta(\r)}{1-\beta(\r)}\right)^\nn \cdot \frac1{2^\nn} + \frac{\a}{2^\nn}.
\end{equation}

Finally, choose the positive constants $\a, \r$ and $\e_\nn$ so that
$\beta(\r)<1/2$,
$$
\D_\nn\cdot\left(\frac{1+\beta(\r)}{1-\beta(\r)}\right)^\nn  + \a + 2^\nn \e_\nn< 1
$$
and
$$
\e_\nn \leq \a\left(\frac1{2^\nn}-\left(\frac1{1+(\r+1)^{1/\nn}}\right)^\nn\right).
$$
Then by (\ref{eq:case1}) and (\ref{eq:case2estimate}) we get that
$\overline{\d}(A) \leq 1 / 2^\nn-\e_\nn$ in both cases, which completes
the proof.
\end{proof}

The theorem above suggests that one should try to prove that any
1-avoiding set of ``high density'' must have block structure. A
natural idea is to take any 1-avoiding set~$A$ and try to modify it in
some manner to obtain a new 1-avoiding set~$\tilde A$ having block
structure and at least the same density as~$A$. Unfortunately, we
could not prove anything rigorous along these lines.

We end this section by presenting an example of a 1-avoiding set of
positive (but small) density which does not have block-structure.
This example also shows that not every 1-avoiding set of positive
density can be modified in a natural way to obtain a new 1-avoiding
set that has block structure and larger or equal density.

Consider the scaled integer lattice $(c\Z)^2\subseteq \RR^2$ with
$c=2\sqrt{2}-2$, and place an open disk of radius $r=(3-2\sqrt{2})/2$
at each lattice point. It is easy to check that the distance between
points of disks around adjacent lattice points is less than $1$ and
the distance between points of disks around nonadjacent lattice
points is bigger than $1$.  Therefore this is a 1-avoiding set without
block structure.
 Simple calculation shows that the the density of this
example is $\delta=r^2 \pi/ c^2\approx 0.0337$. (A somewhat higher density could be achieved with a similar construction using the hexagonal lattice instead of the integer lattice.)


\section{A better upper bound in the plane}
\label{sec:planar}

We now show how a strengthening of the method of Oliveira and
Vallentin~\cite{OliveiraV2010} can be used to provide the bound below.

\begin{theorem} \label{numeric}
\[
m_1(\R^2) \leq 0.258795.
\]
\end{theorem}
Let~$A \subseteq \R^n$ be a measurable and periodic $1$-avoiding set
with periodicity lattice~$L \subseteq \R^n$. Its \textit{autocorrelation
  function} is the function~$f\colon \R^n \to \R$ defined by
\[
f(x) = \delta(A \cap (A - x)).
\]

Determining~$m_1(\R^n)$ is equivalent to the following optimization
problem: find a function~$f$ that maximizes~$f(0)$ and is the
autocorrelation function of a periodic $1$-avoiding set. The
difficulty here lies in the fact that we do not have a
 characterization of autocorrelation functions of
$1$-avoiding sets.  If we give up on finding autocorrelation
functions, but settle for functions satisfying a few of the
constraints that autocorrelation functions do, then we get a
\textit{relaxation} of our original problem and an upper bound
for~$m_1(\R^n)$.  The following lemma gives some such constraints.
Recall  that by $\alpha(G)$ we denote the
independence number of a graph $G$.

\begin{lemma}
\label{lem:constraints}
Let~$f$ be the autocorrelation function of a measurable and periodic
\hbox{$1$-avoiding} set~$A \subseteq \R^n$. Then:

\begin{enumerate}
\item $f(x) = 0$ if~$\|x\| = 1$;

\item if~$G = (V, E)$ is a finite, nonempty subgraph of the
  unit-distance graph of~$\R^n$, then
\[
\sum_{x \in V} f(x) \leq f(0) \alpha(G);
\]

\item if~$C \subseteq \R^n$ is a finite set of points, then
\[
\sum_{\{x,y\} \in {C\choose 2}} f(x-y) \geq |C| f(0) - 1,
\]
where~$C\choose 2$ is the set of all pairs of points in~$C$.
\end{enumerate}
\end{lemma}

\begin{proof}
For~(1) it suffices to observe that, since~$A$ is $1$-avoiding,
if~$\|x\| = 1$ then $A \cap (A - x) = \emptyset$.

For~(2), we claim
that any~$z \in \R^n$ belongs to at most~$\alpha(G)$ of the
sets~$A - x$ for~$x \in V$. Indeed, say~$z$ belongs to all
sets~$A - x_i$ for~$\{x_1, \ldots, x_k\} \subseteq V$. This means that
$a_i-x_i=z$ with some $a_i\in A$. Now if $k > \alpha(G)$, then
there is a pair of points~$x_i$, $x_j$ adjacent in~$G$,
which means $ \|a_i - a_j\| = \|x_i - x_j\| = 1$, a contradiction.


This observation now gives
\[
\delta(A) \geq \delta\biggl(\bigcup_{x\in V} (A \cap (A - x))\biggr)
\geq \alpha(G)^{-1} \sum_{x\in V} \delta(A \cap (A - x)),
\]
as wanted.

Property~(3) is an application of the inclusion-exclusion principle. We have
\[
\begin{split}
1 &\geq \delta\biggl(\bigcup_{x \in C} A - x\biggr)\\
&\geq \sum_{x\in C} \delta(A - x) - \sum_{\{x,y\} \in {C\choose 2}}
\delta((A - x) \cap (A - y))\\
&=|{C}| \delta(A) - \sum_{\{x,y\} \in {C\choose 2}} \delta(A \cap (A -
(x - y))),
\end{split}
\]
and together with the definition of~$f$ we are done.
\end{proof}

Constraint~(2) was observed by Oliveira and Vallentin~\cite{OliveiraV2010}
in the case when the graph~$G$ is a clique. Constraint~(3) was used by
Székely~\cite{Szekely1984}. 

\begin{remark}\rm
The autocorrelation function encodes information about pairs of points
in the set. We may also consider higher-order correlation functions which
 encode information about tuples of points.

For every~$k \geq 0$, consider the function~$F_k\colon (\R^n)^k ×
\{0,1\}^{k+1} \to \R$ given by
\[
F_k((x_1, \dots x_k), (\e_0, \e_1, \dots, \e_k))=\delta(A^{\e_0}\cap
(A-x_1)^{\e_1}\cap \dots \cap (A-x_k)^{\e_k}),
\]
where the notation~$X^1 = X$ and~$X^0 = \R^n \setminus X$ is used.

These functions are nonnegative and satisfy the following
\textit{self-consistency relation}: for all~$k \geq 1$, $x_1$,
\dots,~$x_k$ and~$\e_0$, $\e_1$, \dots,~$\e_{k-1}$ we have
\begin{multline}
  F_{k-1}((x_1, \dots, x_{k-1}),(\e_0, \dots, \e_{k-1}))\\= F_{k}((x_1,
  \dots, x_{k-1}, x_k),(\e_0, \dots, \e_{k-1}, 0))\\+F_{k}((x_1,
  \dots, x_{k-1}, x_k),(\e_0, \dots, \e_{k-1}, 1)).
\end{multline}

Also, similarly to property (1) in the lemma above, we have
\[
F_{k}((x_1, \dots, x_k),(\e_0, \e_1, \dots, \e_{k}))=0
\]
whenever there exist indices $i$, $j$ such that $\|x_i-x_j\|=1$ and
$\e_i=\e_j=1$. This property and the self-consistency
relations together imply properties~(2) and~(3) in the lemma above (we omit
the proof of this fact). Therefore, one may regard the
self-consistency of higher-order correlation functions as the
common source of properties~(2) and~(3).\hfill$\diamond$
\end{remark}

\begin{remark}\rm
Using Kai Lai Chung's inequalities as Sz\'ekely and Wormald did in \cite{SzekelyW1989}, one can obtain a generalization of property (3) in Lemma \ref{lem:constraints}. Namely, for every $m\ge 1$ we have
$\sum_{\{x,y\} \in {C\choose 2}} f(x-y) \geq m|C| f(0) - \binom{m+1}{2}$ (and, in particular, $m=1$ yields property (3)). In principle, these inequalities could be used to further improve the upper bound on $m_1(\R^2)$, as was the case in higher dimensions in \cite{SzekelyW1989}. However, when we implemented the inequalities in the computer code (cf. the description of section \ref{informal} below), they did not yield any numerical improvement. \hfill$\diamond$
\end{remark}

In order to construct an optimization problem in which an
autocorrelation function of a $1$-avoiding set is to be found,
we shall parametrize such functions
via their Fourier series. This parametrization will also
suggest one further constraint satisfied by autocorrelation
functions. All the facts we state from harmonic analysis are quite
standard; the reader looking for reference is advised to consult the book by
Katznelson~\cite{Katznelson1968}.

Given measurable functions~$f$, $g\colon \R^n \to \C$, write
\begin{equation}
\label{eq:inner-p}
\langle f, g\rangle = \lim_{T \to \infty} \frac{1}{(2T)^n}
\int_{[-T,T]^n} f(x) \overline{g(x)}\, dx,
\end{equation}
when the limit exists.


Let~$L \subseteq \R^n$ be a lattice. A function~$f\colon \R^n  \to\C$ is \textit{periodic}, if
it is invariant under the action of~$L$: we have~$f(x + v) = f(x)$ for
all~$x \in \R^n$ and~$v \in L$. Lattice~$L$ is a \textit{periodicity
  lattice} of~$f$. Such functions are in a natural 1-1 correspondence with
functions $f\colon \R^n / L \to\C$; with an abuse of notation we will use the same
letter for both.

Now~\eqref{eq:inner-p} defines an inner product in the space of
square-integrable, complex-valued functions with periodicity
lattice~$L$, isomorphic to $L^2(\R^n / L)$. Equipped with this inner
product,~$L^2(\R^n / L)$ is a Hilbert space. Functions
\[
\chi_u(x) = e^{i u \cdot x}\qquad\text{for $u \in 2\pi L^*$,}
\]
where~$L^* = \{\, u \in \R^n : \text{$u\cdot v \in \Z$ for all~$v \in
  L$}\,\}$ is the \textit{dual lattice} of~$L$, form a complete
orthonormal system of~$L^2(\R^n / L)$.

The Fourier coefficient of~$f$ at~$u \in 2\pi L^*$
is~$\widehat{f}(u) = \langle f, \chi_u\rangle$.  Since the~$\chi_u$
form a complete orthonormal system, we have the Fourier inversion
formula
\[
f(x) = \sum_{u \in 2\pi L^*} \widehat{f}(u) e^{i u \cdot x},
\]
with~$L^2$ convergence.

Let~$A \subseteq \R^n$ be a measurable set with periodicity
lattice~$L$. We denote by~$\one_A $ the indicator function of~$A$, an $L$-periodic function.
 The autocorrelation function of~$A$ is
\[
f(x) = \langle \one_A, \one_{A - x} \rangle.
\]
Since~$f(x)$ is expressed in terms of the inner
product~\eqref{eq:inner-p}, and since~$\widehat{\one}_{A - x}(u) =
\widehat{\one}_A(u) e^{i u \cdot x}$, Parseval's identity gives
\begin{equation}
\label{eq:f-exp}
f(x) = \sum_{u\in 2\pi L^*} |\widehat{\one}_A(u)|^2 e^{i u \cdot x}
\end{equation}
with convergence for all~$x$. The inversion formula shows
that~$\widehat{f}(u) = |\widehat{\one}_A(u)|^2$. This gives
another constraint satisfied by an autocorrelation function, namely
that its Fourier coefficients are all nonnegative, in other words,
$f$ is \textit{positive definite}.

 We will \textit{radialize}~$f$ by averaging
 over the sphere or, equivalently, over the orthogonal group. In
other words, we set
\begin{equation}
\label{eq:radial}
\mathring{f}(x) = \frac{1}{\omega(S^{n-1})} \int_{S^{n-1}} f(\xi
\|x\|)\, d\omega(\xi) = \int_{\ort(\R^n)} f(T x)\, d\mu(T),
\end{equation}
where~$\omega$ is the surface measure on the unit
sphere~$S^{n-1} \subseteq \R^n$ and~$\mu$ is the normalized  Haar measure on the
orthogonal group~$\ort(\R^n)$.

This function $\mathring{f}$ is \textit{radial}, i.e., the value
of~$\mathring{f}(x)$ depends only on~$\|x\|$. (It is not periodic any more,
nor is it typically the autocorrelation function of a set.)

 Let~$\Omega_n$ be the
function defined over the nonnegative reals which is such that
\[
\Omega_n(\|x\|) = \frac{1}{\omega(S^{n-1})} \int_{S^{n-1}} e^{i x
  \cdot \xi}\, d\omega(\xi)
\]
for all~$x \in \R^n$.
Then, using the inversion formula we get
\[
\mathring{f}(x)=\frac{1}{\omega(S^{n-1})}
\int_{S^{n-1}} \sum_{u \in 2\pi L^*}
\widehat{f}(u) e^{i u \cdot \xi\|x\|}\, d\omega(\xi)=
\sum_{u \in 2\pi L^*}\widehat{f}(u)\Omega_n( \|u\| \|x\|).
\]
We can rewrite this as
\begin{equation}
\label{eq:rad-exp}
\mathring{f}(x) = \sum_{t \geq 0} \kappa(t) \Omega_n(t \|x\|),
\end{equation}
where~$\kappa(t)$ is the sum of~$\widehat{f}(u)$ over all~$u$ such
that~$\|u\| = t$. The sum in~\eqref{eq:rad-exp} is absolutely convergent:
 only countably many of the~$\kappa(t)$
coefficients are nonzero, they are nonnegative, their sum is $\delta(A)$ and the $\Omega$'s are bounded.

All constraints in Lemma~\ref{lem:constraints} are
rotation-invariant (inequality (2) for a fixed graph, or inequality (3) for a fixed set is not
invariant, but the property that they hold for all graphs or sets, respectively, is).
As the autocorrelation function of a
$1$-avoiding set satisfies these constraints, so does its
radialization. We list these properties in terms of the function $\kappa$.
Write $\delta=\delta(A)=f(0)$.

\begin{equation} \label{eq:prima}
\begin{array}{rll}
&\sum_{t \geq 0} \kappa(t) = \delta,\\[1ex]
&\sum_{t \geq 0} \kappa(t) \Omega_n(t) = 0,\\[1ex]
&\sum_{t \geq 0} \kappa(t) \sum_{x \in V(G)} \Omega_n(t\|x\|) \leq
 \delta \alpha(G)&\text{for all graphs $G$},\\[1ex]
&\sum_{t \geq 0} \kappa(t) \sum_{\{x,y\} \in {C\choose 2}}
  \Omega_n(t\|x-y\|) \geq \delta|C| - 1&\text{for all finite $C \subset \R^n$}.\\[1ex]
\end{array}
\end{equation}

 We also know that
\begin{equation} \label{eq:secunda}
\begin{array}{rll}
& \kappa(0)=\delta^2, \\
&\kappa(t) \geq 0&\text{for all~$t \geq 0$}.
\end{array}
\end{equation}

In the sequel we will use this (infinite) system of inequalities to get an
upper bound for $\delta$ and hence for $m_1(\R^n)$. We will use only a finite number of graphs and point configurations,
 in order to be able to check the properties of the "witness" function $W(t)$ described below.
The above inequalities are necessary but very likely
not sufficient for a funtion to be the radialization of an autocorrelation function, so probably this approach cannot
yield the best bound.

\begin{proposition}
\label{thm:lp-bound}
  Let~$\Scal$ be a finite collection of finite subgraphs of the
  unit-distance graph of~$\R^n$ and let~$\Ccal$ be a finite collection
  of finite sets of points in~$\R^n$. Suppose that the numbers~$v_0$,
  $v_1$, $w_G \geq 0$ for~$G \in \Scal$, and~$z_C \geq 0$
  for~$C \in \Ccal$ are such that the function
\begin{equation}\label{witness}
W(t) = v_0 + v_1 \Omega_n(t) + \sum_{G \in \Scal} w_G \sum_{x \in V(G)} \Omega_n(t\|x\|)-
\sum_{C \in \Ccal} z_C \sum_{\{x,y\} \in {C\choose 2}} \Omega_n(t\|x-y\|)
\end{equation}
satisfies $W(0) \geq 1$ and $W(t) \geq0$ for $t>0$.
Then $m_1(\R^n) \leq \delta$ where $\delta$ is the solution of the equation
\begin{equation}\label{deltaup}
\delta^2 = \delta \left( v_0 + \sum_{G \in \Scal} w_G \alpha(G) - \sum_{C \in \Ccal}
z_C  |C|  \right)\ +   \sum_{C \in \Ccal} z_C.
\end{equation}
\end{proposition}

\begin{proof}

With any function $W$ satisfying
$W(0) \geq 1$ and $W(t) \geq0$ for $t>0$ we have
\begin{equation}\label{alpha0}
 \delta^2 = \kappa(0) \leq \sum_{t \geq 0} \kappa(t) W(t).
\end{equation}
If $W$ is in the form \eqref{witness} then inequalities \eqref{eq:prima} and \eqref{alpha0} imply
\begin{equation}\label{deltaup1}
\delta^2 \leq \delta \left( v_0 + \sum_{G \in \Scal} w_G \alpha(G) - \sum_{C \in \Ccal}
z_C  |C|  \right)\ +   \sum_{C \in \Ccal} z_C.
\end{equation}
\end{proof}

\subsection{Applying Proposition~\ref{thm:lp-bound} for the Euclidean plane}\label{informal}

In this section we will explain informally how one can look for
good collections~$\Scal$ and~$\Ccal$. Our approach for this is
experimental.  We then give explicit values for~$v$, $w$, and~$z$ in subsection \ref{values}, and
show how it can be verified that the conditions of
Proposition~\ref{thm:lp-bound} are satisfied.

As we will need to work with the function $\Omega_2(t)$, recall that ~$\Omega_n$ has an expression in terms of Bessel
functions, namely
\[
\Omega_n(t) = \Gamma\Bigl(\frac{n}{2}\Bigr)
\Bigl(\frac{2}{t}\Bigr)^{(n-2)/2} J_{(n-2)/2}(t)
\]
for~$t > 0$ and~$\Omega_n(0) = 1$, where~$J_\alpha$ is the Bessel
function of the first kind with parameter~$\alpha$; this formula was
first observed by Schoenberg~\cite{Schoenberg1938}.

In trying to apply Proposition \ref{thm:lp-bound} we can start with ~$\Scal = \Ccal = \emptyset$. In that case $W(t)= v_0 + v_1 \Omega_2(t)$, and the best upper bound that can be achieved this way is~$\approx 0.287$, which is slightly worse than the bound of~$2/7 \approx 0.285$
coming from the Moser spindle (cf. the Introduction). Oliveira and Vallentin~\cite{OliveiraV2010} took ~$\Ccal = \emptyset$ and ~$\Scal$ to be a few equilateral triangles in~$\R^2$ (at appropriate positions). This
provided an upper bound of~$\approx 0.268$, which was better than the
previously known best upper bound of~$\approx 0.279$ due to
Székely~\cite{Szekely1984}.

A further improvement can be obtained if one takes ~$\Ccal = \emptyset$ and ~$\Scal$ to be a few
congruent copies of the Moser spindle at appropriate positions -- as explained here. Let~$G$ be the Moser spindle as
shown in Figure~\ref{fig:moser}, with the lower-left vertex placed at
the origin. Consider congruent copies of~$G$ of the form
\begin{equation}
\label{eq:mosers}
(t, 0) + R(\theta) G,
\end{equation}
where~$t \in \R$, $\theta \in [0, 2\pi]$, and
\[
R(\theta) = \begin{pmatrix}
\cos \theta &-\sin \theta \\
\sin \theta &\cos \theta
\end{pmatrix}
\]
is a rotation matrix. In order to get a finite number of copies we discretize the values of ~$t$ and ~$\theta$. Let $\e=0.1$ and consider $-4\leq t=j\e\leq 4$, and $0\leq \theta=k\e \leq 2\pi$ for $j,k\in \Z$. As a first step we take
~$\Scal$ to be all these copies of~$G$.

Instead of looking for a witness function $W(t)$ directly, we will consider the
system of inequalities \eqref{eq:prima} and \eqref{eq:secunda}. We introduce the normalized variables
$\tilde{\kappa}(t)=\kappa(t)/\delta$, and write the system \eqref{eq:prima}, \eqref{eq:secunda} as a linear programming problem (of infinitely many variables $\tilde\kappa(t)$) for maximizing $\tilde\kappa(0)$:

\begin{equation}
\label{eq:primal}
\begin{array}{rll}
\sup&\tilde\kappa(0)\\
&\sum_{t \geq 0} \tilde\kappa(t) = 1,\\[1ex]
&\sum_{t \geq 0} \tilde\kappa(t) \Omega_n(t) = 0,\\[1ex]
&\sum_{t \geq 0} \tilde\kappa(t) \sum_{x \in V(G)} \Omega_n(t\|x\|) \leq
 \alpha(G)&\text{for~$G \in \Scal$},\\[1ex]
&\sum_{t \geq 0} \tilde\kappa(t) \sum_{\{x,y\} \in {C\choose 2}}
  \Omega_n(t\|x-y\|) \geq |C| - \delta^{-1}&\text{for~$C \in
                                             \Ccal$},\\[1ex]
&\tilde\kappa(t) \geq 0&\text{for all~$t \geq 0$}.
\end{array}
\end{equation}

If for any value of $\delta$ we obtain $\sup \tilde\kappa(0)=\rho$ as the solution of this linear programming problem, and $\rho\leq \delta$, then we should be able to find a witness function $W(t)$ in the form \eqref{witness} by linear duality. In trying to solve this primal problem~\eqref{eq:primal} an
issue is that we have infinitely many variables $\tilde\kappa(t)$. The workaround
is to pick numbers~$L > 0$ and~$\varepsilon > 0$ and then only use
variables~$\kappa(t)$ for~$t$ of the form~$k\varepsilon \leq L$
with~$k \geq 0$ integer.

By taking~$L = 200$,~$\varepsilon = 0.01$ and~$\delta= 0.26305$ we get~$\sup \tilde\kappa(0)\approx\delta=0.26305$. Therefore, a witness function
$W(t)= v_0 + v_1 \Omega_n(t) + \sum_{G \in \Scal} w_G \sum_{x \in V(G)} \Omega_n(t\|x\|)$ testifying $m_1(\R^n) \lessapprox 0.26305$ should exist by linear duality (in fact, we must expect a little loss because of the discretization in the values of $t$).
Note that this bound is already better than that
of Oliveira and Vallentin~\cite{OliveiraV2010}, ~$m_1(\R^n) \leq 0.268$, using equilateral triangles. Also, from the solution of the linear
program \eqref{eq:primal} we can see that only a few of the constraints coming from
copies of the Moser spindle are actually used; the rest may be discarded. In fact we will only keep three Moser spindles, as given in \eqref{spindles} below.

Now we can try to add inclusion-exclusion constraints. (Note that we have not
produced any witness function $W(t)$ yet, we have been working with the variables $\tilde\kappa(t)$. We will keep doing so, and only turn to
constructing $W(t)$ at the end.)

Let~$\tilde\kappa(t)$ be
an optimal solution of the discretized linear program above (containing the constraints coming from the three copies of the Moser spindle given in \eqref{spindles}). Temporarily, we will fix these values $\tilde\kappa(t)$ and fix~$N>0$.
(In our
experiments the only value that led to improvements was ~$N = 6$.)
We want to find a set~$C$ of~$N$ points in~$\R^2$ which
minimizes the left-hand side of the inclusion-exclusion constraint
in~\eqref{eq:primal}. Due to translation invariance, we can assume
that the origin belongs to~$C$. We can consider the coordinates of the other~$N-1$
points as variables, and we try to
minimize the left-hand side of the constraint using
some numerical nonlinear optimization method. This will give us a set~$C$ of points. To the right hand side we substitute the actual value of~$\delta=0.26305$, and see whether the constraint is violated. If it is, then the inclusion of $C$ in $\Ccal$ will give us an improved bound on $m_1(\R^n)$. Namely, we add this new constraint to the linear programming problem \eqref{eq:primal}, and find the value of $\delta$ such that the optimum satisfies $\sup \tilde\kappa(0)=\delta$ (note here that $\sup \tilde\kappa(0)$ depends on $\delta$ because the right hand side of the added constraint depends on $\delta$). By linear duality, we can then hope to find a witness $W(t)$ which testifies $m_1(\R^n)\lessapprox \delta$ (again, a little loss must be expected because of the discretization in $t$).

This procedure can be iterated: we consider an optimal solution $\tilde\kappa(t)$, and
find another set~$C$ of points that minimizes the left-hand side of
the inclusion-exclusion constraint. And we repeat. After a few
iterations, we start to reach the limit of this approach.

Finally we get a dual solution of the resulting linear
program~\eqref{eq:primal}, which gives us the numbers~$v_0$,
$v_1$, $w_G$ and~$z_C$ to be substituted in
Proposition~\ref{thm:lp-bound}. These numbers will nearly satisfy the conditions of the theorem if our
discretization was fine enough. After some minor adjustments in these values (incurring a little loss in the value of $\delta$) we can
actually verify that they satisfy the conditions, and we obtain a valid bound for $m_1(\R^n)$.

\subsection{Explicit values and verification}\label{values}

We now present candidate values for~$v$, $w$, and~$z$ for Proposition~\ref{thm:lp-bound},
 and then show how to modify these values
slightly so that they satisfy the conditions of the theorem.

For us,~$\Scal$ is composed of three copies of Moser's spindle,~$G_1$,
$G_2$, and~$G_3$ given as in~\eqref{eq:mosers}, corresponding to the
following pairs~$(t, \theta)$:
\begin{equation}\label{spindles}
\text{$(0.4, 5.4)$, $(0.6, 5.4)$, and~(0.8, 5.4)}.
\end{equation}
The collection~$\Ccal$ we use is given in Table~\ref{tab:ccal}.

\begin{table}[htb]
\bgroup
\setbox0=\hbox{$(0)$}
\newdimen\h \h=\ht0
\newdimen\d \d=\dp0

\advance\h by 6pt
\advance\d by 6pt

\def\hr{\vrule height\h width0pt depth0pt}
\def\dr{\vrule height0pt width0pt depth\d}
\hbox to\hsize{\hss
\begin{tabular}{ll|ll}
\hline
\hr&$(\phantom{-}0.781846561681, \phantom{-}0.923983014983)$&&$(\phantom{-}0.533352656963, \phantom{-}0.891484779083)$\\
&$(-1.493218191370, \phantom{-}0.715876600816)$&&$(-0.611400296245, \phantom{-}0.779442549608)$\\
$C_1$&$(-0.413794012440, \phantom{-}0.699470697606)$&$C_4$&$(-0.285536136585, \phantom{-}1.699218505820)$\\
&$(-1.195640884520, -0.224511807288)$&&$(\phantom{-}0.251297714466, \phantom{-}0.991412992863)$\\
\dr&$(\phantom{-}1.079423910680, -0.016405441239)$&&$(\phantom{-}0.680196169514, \phantom{-}1.919904450610)$\\
\hline
\hr&$(\phantom{-}0.976422451180, \phantom{-}0.219342709492)$&&$(\phantom{-}0.665906520384, \phantom{-}2.751699047290)$\\
&$(-0.896557239530, \phantom{-}0.403173339690)$&&$(-1.358694685180, \phantom{-}1.253666844760)$\\
$C_2$&$(-0.552919373209, \phantom{-}1.316137405620)$&$C_5$&$(-0.448270339088, \phantom{-}0.880354589520)$\\
&$(\phantom{-}0.051861274857, \phantom{-}0.567345039740)$&&$(-0.943967767036, \phantom{-}0.337966779380)$\\
\dr&$(\phantom{-}0.386966521215, \phantom{-}1.028078255990)$&&$(\phantom{-}1.397952081510, \phantom{-}2.088390922180)$\\
\hline
\hr&$(\phantom{-}0.951509148625, \phantom{-}0.297382071175)$&&\\
&$(-0.856129318724, \phantom{-}0.498561149113)$&&\\
$C_3$&$(-0.613074338850, \phantom{-}1.455125134570)$&&\\
&$(\phantom{-}0.035657780606, \phantom{-}0.706699050884)$&&\\
\dr&$(\phantom{-}0.297303535201, \phantom{-}1.051961996200)$&&\\
\hline
\end{tabular}
\hss}
\egroup
\bigskip

\caption{Collection~$\Ccal$ of point-sets used.
   Each of the sets also contains the
  origin~$(0, 0)$, so each set has~$6$ points.}
\label{tab:ccal}
\end{table}

We set the values of~$v$, $w$, and~$z$ to:
\begin{equation}
\label{eq:sol}
\begin{array}{lr}
v_0& 2.3022516897351055\\
v_1&27.2729338671989154\\
w_1& 0.2021538298582705\\
w_2& 0.4311844458316473\\
w_3& 1.3855315999360112\\
z_1& 0.2862826361013497\\
z_2& 0.7908579212800153\\
z_3& 0.9616086568833265\\
z_4& 0.2772120180959884\\
z_5& 0.5311904133936868
\end{array}
\end{equation}
Here,~$w_i$ is associated with graph~$G_i$ in~$\Scal$, and
similarly~$z_i$ is associated with configuration~$C_i \in \Ccal$.

Recall that~$\Omega_2(t) = J_0(t)$ and let
\[
\varphi(t) = v_0 + v_1 J_0(t) + \sum_{i=1}^3 w_i \sum_{x
  \in V(G_i)} J_0(t\|x\|)
-\sum_{i=1}^5 z_i \sum_{\{x,y\} \in {C_i\choose 2}}
J_0(t\|x-y\|).
\]
The conditions required of~$v$, $w$, and~$z$ in
Proposition~\ref{thm:lp-bound} now translate to~$\varphi(0) \geq 1$
and~$\varphi(t) \geq 0$ for all~$t > 0$.

It is easy to check that~$\varphi(0) \geq 1$. To show
that~$\varphi(t) \geq 0$ for all~$t > 0$, the first step is to notice
that
\[
\lim_{t\to\infty} J_0(t) = 0.
\]
This follows from the asymptotic formula for~$J_\alpha$ for~$\alpha
\geq 0$ (cf.~Watson~\cite{Watson1922}, equation~(1) in §7.21). So we see
that
\[
\lim_{t\to\infty} \varphi(t) = v_0,
\]
and so~$\varphi(t) \geq 0$ for all large enough~$t$.

We need an estimate on how large~$t$ has to be
chosen
and this we can
get as follows. We have
\[
\frac{d J_0(t)}{dt} = -J_1(t).
\]
Let~$j_1 < j_2 < j_3 < \cdots$ be the positive zeros of~$J_1$. By
the above expression for the derivative, these are the places of  local extrema
of~$J_0$.  The local extrema of~$J_0$ decrease in absolute value
(cf.~Watson~\cite{Watson1922}, §15.31), that is
\[
|J_0(j_1)| > |J_0(j_2)| > |J_0(j_3)| > \cdots.
\]
 Hence to find an upper bound on the absolute value
of~$J_0(t)$ for~$t \geq L$ it is sufficient to find the rightmost
zero of~$J_1$ in the interval~$[0, L]$ and compute~$J_0$ at this
zero. There are procedures to compute the zeros of~$J_1$ to any
desired precision.

Using this idea, we may check that for~$\Scal$ and~$\Ccal$ as we have
the absolute value of
\[
v_1 J_0(t) + \sum_{i=1}^3 w_i \sum_{x
  \in V(G_i)} J_0(t\|x\|)
-\sum_{i=1}^5 z_i \sum_{\{x,y\} \in {C_i\choose 2}}
J_0(t\|x-y\|)
\]
for~$t \geq 779.8998\ldots$ (this is the 248th positive zero of~$J_1$)
is at most~$v_0 - 0.05 \approx 2.2522$. Consequently~$\varphi(t) \geq
0$ for all~$t \geq L$ with~$L = 780$.

Now we  check whether~$\varphi(t) \geq 0$ in~$[0, L]$. This will
not be the case: since our solution has been found numerically via
sampling, it will be negative at some points. But it will only be
slightly negative, and then adding a small number to~$v_0$ will make
it nonnegative everywhere.


Recall that the derivative of~$J_0$ is~$-J_1$.
Since~$|J_1(t)| \leq 1/\sqrt{2}$ for all~$t \geq 0$
(cf.~Watson~\cite{Watson1922}, equation~(10) in~§13.42), we can provide a
rough estimate for~$|\varphi'(t)|$, namely
\[
|\varphi'(t)| \leq 75.9547\qquad\text{for all~$t \geq 0$}.
\]
Then the mean-value theorem implies that
\[
|\varphi(t_1) - \varphi(t_2)| \leq 75.9547 |t_1 - t_2|
\]
for every~$t_1$, $t_2 \geq 0$.

If for a prescribed~$\varepsilon > 0$ we compute~$\varphi(t)$ for
all~$t = k \varepsilon / 76 \leq L$ with~$k \geq 0$ integer and take
the minimum, this gives the
minimum of~$\varphi$ in~$[0, L]$ up to an additive error
of~$\varepsilon$.

Taking~$\varepsilon = 10^{-4}$, we obtain the conservative estimate
that the minimum of~$\varphi$ in~$[0, L]$ is at
least~$-0.00011$. Adding this to~$v_0$ we then get numbers~$v$, $w$,
and~$z$ that satisfy the conditions of Proposition~\ref{thm:lp-bound}.
  Solving the quadratic inequality we obtain
$\delta \leq 0.258795$, an upper bound
for~$m_1(\R^2)$.

We have attempted to include more constraints from Moser spindles
after the 6-point configurations were added. This improves the bound slightly, but
not much. Attempts to add more 6-point configurations have run into
numerical trouble, and we could not derive a rigorous bound from such
trials. They suggest that better bounds can be achieved, but we never
managed to get below~$0.257$, which is probably the limit of this method.

The verification procedure we just described was implemented in a
Sage~\cite{SAGE} script that is available together with the arXiv
version of this paper. It is a short program that can be easily
checked by the reader. Numerical computations are still carried out to
compute Bessel functions, but due to the simplicity of the code, one
can have a high degree of confidence on the results obtained. A fully
rigorous verification procedure would require the use of rational
arithmetic and this would require Bessel functions to be computed up
to good precision using rationals. This is not hard to
implement, but in our view it is not necessary for the present result.

\medskip

\begin{center}
Acknowledgement
\end{center}

The authors are grateful to the referees whose valuable suggestions helped to improve the presentation of the paper.

\end{document}